\documentclass[12pt]{amsart}
\usepackage[utf8]{inputenc}

\title[Gorenstein Hierarchical Models]{The Codegree, Weak Maximum Likelihood Threshold, and the Gorenstein Property of Hierarchical Models}
\author{Joseph Johnson and Seth Sullivant}
\date{\today}

\usepackage{amssymb}
\usepackage{amsfonts}
\usepackage{amsthm}
\usepackage{amsmath}
\usepackage{amscd}
\usepackage{t1enc}
\usepackage[mathscr]{eucal}
\usepackage{indentfirst}
\usepackage{graphicx}
\usepackage{graphics}
\usepackage{pict2e}
\usepackage{epic}
\numberwithin{equation}{section}
\usepackage[margin=2.9cm]{geometry}
\usepackage{epstopdf} 
\usepackage{tikz}
\usepackage{transparent}
\usepackage{wasysym}
\usepackage{enumerate}

\usepackage{mathdots}
\usepackage{todonotes}
\usepackage{outlines}

\usepackage{caption}

 \colorlet{myGreen}{green!50!gray!120!}

\theoremstyle{plain}
\newtheorem{Th}{Theorem}[section]
\newtheorem{Lemma}[Th]{Lemma}
\newtheorem{Cor}[Th]{Corollary}
\newtheorem{Prop}[Th]{Proposition}

 \theoremstyle{definition}
\newtheorem{Def}[Th]{Definition}
\newtheorem{Conj}[Th]{Conjecture}

\newtheorem{Rem}[Th]{Remark}
\newtheorem{?}[Th]{Problem}
\newtheorem{Ques}[Th]{Question}
\newtheorem{Ex}[Th]{Example}

\newcommand{\calm}{\mathcal{M}}
\newcommand{\calp}{\mathcal{P}}
\newcommand{\calq}{\mathcal{Q}}

\newcommand{\zz}{\mathbb{Z}}
\newcommand{\rr}{\mathbb{R}}

\DeclareMathOperator{\cone}{cone}
\DeclareMathOperator{\conv}{conv}
\DeclareMathOperator{\wmlt}{wmlt}
\DeclareMathOperator{\Marg}{Marg}

\tikzstyle{wB}=[circle, draw, fill=black, inner sep=0pt, minimum width=4.5pt]

\begin{document}

\maketitle

\begin{abstract}
The codegree of a lattice polytope is the smallest integer dilate that contains a lattice 
point in the relative interior.   
The weak maximum likelihood threshold of a statistical model
is the smallest number of data points for which there is a non-zero probability
that the maximum likelihood estimate exists.  
The codegree of a marginal polytope is a lower bound on the maximum likelihood threshold
of the associated log-linear model, and they are equal when the marginal polytope
is normal.
We prove a lower bound on the codegree in the case of hierarchical log-linear models
and provide a conjectural formula for the codegree in general.
As an application, we study when the marginal polytopes of hierarchical models
are Gorenstein, including a classification of Gorenstein decomposable models,
and a conjectural classification of Gorenstein binary hierarchical models.



\end{abstract}



    



\section{Introduction}

Let $P$ be an integral polytope in $\mathbb{R}^d$.  The $k$th dilate
of $P$ is $kP =  \{ k x  : x \in P\}$.  
The \emph{codegree} of $P$ is the smallest integer $k$ such
that $kP$ contains a lattice point in its relative interior.
The codegree is a fundamental invariant of a lattice polytope, related to the
structure of the Ehrhart series of that polytope, which encodes lattice point
enumeration properties of the polytope.

Given a statistical model $\mathcal{M}$, and some data $D = D_1, \ldots, D_n $, 
a common problem is to compute
the maximum likelihood estimate, which is the point in the model that maximizes the likelihood
that the given data set $D$ was observed.  If the number of data points $n$ is
too small, it is possible that the maximum-likelihood estimate might not
exist.  The \emph{weak maximum likelihood threshold} of $\mathcal{M} $ is the smallest
$n$ for which the probability that the maximum likelihood estimate exists is greater than $0$.
The weak maximum likelihood threshold was first studied in the context of
Gaussian graphical models where it was related to combinatorial concepts like the
chromatic number of a graph and rigidity theory \cite{GrossSullivant2018, Bernstein2021}.

In this paper, we study the codegrees of polytopes associated to log-linear models.
A key observation is that the codegree of the marginal polytope of a log-linear
model is a lower bound on the weak maximum likelihood threshold of the model,
and, in fact, the codegree equals the weak maximum likelihood threshold when
the underlying marginal polytope is normal.

We apply and study the codegree of log-linear models in the specific case of 
hierarchical log-linear models.  We prove a general lower bound for the codegree that
we conjecture to equal the codegree.  We prove this conjecture in many cases.
As an application 
we study when hierarchical models are Gorenstein.  We give a complete classification
in the case of decomposable hierarchical models, using results about how the Gorenstein
property is preserved under the fiber product of polytopes.  We also provide a conjecture
about the structure of Gorenstein hierarchical models in the case of binary random variables.

This paper is organized as follows.  In the next section we give background results
on log-linear models, their marginal polytopes, and key background on lattice point polytopes and
the codegree.  We also prove the connection between the codegree and the
weak maximum likelihood threshold.  In Section \ref{sec:codeghier}, we prove our main results on 
the codegree of hierarchical models.  In Section \ref{sec:Gorenstein}, 
we study the fiber product construction
and how it preserves the Gorenstein property and use that to classify Gorenstein decomposable
models.   We also provide a conjecture about the structure of Gorenstein binary hierarchical models,
and some computational evidence.




\section{Marginal polytopes, codegree, and the weak maximum likelihood threshold}\label{sec:background}

In this section we will discuss the background on marginal polytopes of log-linear
models,
the concept of the codegree of a lattice polytope, and the connection to the
weak maximum likelihood threshold for log-linear models.

Let $[n] = \{1,2,\dots,n\}$. The \emph{probability simplex} $\Delta_{n-1}$ on $[n]$ states  
is the set 
\[
\Delta_{n-1} =  \left\{ p \in \mathbb{R}^n :  \sum_{i =1}^n p_i = 1, \mbox{ and } p_i \geq 0 \mbox{ for all } i \in [n] \right\}.
\]

\begin{Def}
Let  $A \in \mathbb{Z}^{k \times n}$ be an integer
matrix and we assume that $A$ is normalized so that there exists a $w \in \mathbb{Q}^k$ so that
$w^T A  =  {\bf 1}$.
The \emph{log-linear model} associated to $A$ is
\[\mathcal{M}_A = \{p \in \Delta_{n-1}: \log(p) \in \text{rowspan}(A)\}.\]
\end{Def}

The matrix $A$ is called the \emph{design matrix} of the log-linear model. 
The model $\mathcal{M}_A$ is a family of probability distributions for
random variables with $n$ states.  So an element $p \in \mathcal{M}_A$
satisfies $p_i = P(X  = i)$.  

Maximum likelihood estimation concerns the problem of taking data
and computing the probability distribution in the model that maximizes the
probability that the given data occurred. If we assume that the 
data is independent and identically distributed, then for a log-linear model
$\calm_\mathcal{A}$, data is summarized by a vector $u \in \mathbb{N}^n$, called a
vector of counts.  The \emph{sample size} of this data is $\sum_{i = 1}^n u_i$.
Maximum likelihood estimation is the following optimization problem:
\[
\mathrm{argmax}  \prod_{i = 1}^n p_i^{u_i}  \mbox{ subject to }  p \in \mathcal{M}_A.
\]

Birch's theorem changes this optimization problem into the problem of solving a
system of polynomial equations.

\begin{Th}[Birch's Theorem]
Suppose that $\mathbf{1} \in \text{rowspan}(A)$ and let $u \in \mathbb{N}^n$. Then the maximum likelihood estimate is the unique solution to the system
\[Au = (u_1 +\dots+ u_n)Ap \quad and \quad p \in \mathcal{M}_A\]
if it exists.
\end{Th}

A key observation here is the importance of the quantity $Au$ in the solution of
the maximum likelihood estimation problem.  The vector $Au$ is called the vector
of sufficient statistics, or, marginals in the context of hierarchical models (to be explained later).
This leads to the notion of the marginal cone and the marginal polytope
of the mode.

\begin{Def}
Let $A \in \mathbb{Z}^{k \times n}$ be an integer
matrix with  $\mathbf{1} \in \text{rowspan}(A)$.
The \emph{marginal cone} 
is the set $\mathrm{cone}(A)  =  \{  Au:  u \in \mathbb{R}^n_{\geq 0} \}$
which is the cone over all marginals of all possible data vectors $u$.
The \emph{marginal polytope} is the polytope $\mathrm{Marg}(A)$
consisting of the convex hull of the columns of $A$.
\end{Def}

Note that the log-linear model $\mathcal{M}_A$ is an open set, because it does
not include any points on the boundary of $\Delta_{n-1}$.  So in the statement
of Birch's theorem, we have the key caveat that the maximum likelihood estimate
might not exist.  In fact, the existence of the maximum likelihood estimate 
is characterized in terms of the marginal cone (see \cite[Thm 8.2.1]{Sullivant2018}).

\begin{Th}\label{thm:haberman}
Let $A \in \mathbb{Z}^{k \times n}$ be an integer
matrix with  $\mathbf{1} \in \text{rowspan}(A)$. 
Let $u \in \mathbb{N}^u$ be a vector of counts.  Then the maximum likelihood estimate
exists in the log-linear model $\mathcal{M}_A$ if and only if $u$ is in the relative
interior of $\mathrm{cone}(A)$.
\end{Th}

From here, we can explain the concept of the weak maximum likelihood threshold.

\begin{Def}
Let $A \in \mathbb{Z}^{k \times n}$ be an integer
matrix with  $\mathbf{1} \in \text{rowspan}(A)$.
The \emph{weak maximum likelihood threshold} of $\mathcal{M}_A$ (denoted  $\wmlt(\mathcal{M}_A)$) is the
smallest integer $m$ such that there exists a data vector $u \in \mathbb{N}^n$ with
sample size $m = \sum_{i = 1}^n u_i$ and such that the maximum likelihood estimate exists.

Equivalently, by Theorem \ref{thm:haberman}, $\wmlt(\calm_A)$ is the smallest integer $m$
such that there is a $u \in \mathbb{N}^n$ such that $\sum u_i = m$ and $Au$ is in the 
relative interior of $\mathrm{cone}(A)$.  
\end{Def}

The weak maximum likelihood threshold is the smallest number of data points for
which there is a chance that the maximum likelihood estimate might exist. For any
amount of data smaller than $\wmlt(\calm_A)$ the maximum likelihood estimate will
not exist.  

\begin{Rem} 
The weak maximum likelihood threshold was originally introduced in
the context of Gaussian graphical models, where the  maximum likelihood
threshold is also studied \cite{GrossSullivant2018}. 
The maximum likelihood threshold is the smallest
sample size for which the maximum likelihood estimate exists with probability one.
Unfortunately, the maximum likelihood threshold does not exist for models
on discrete random variables, because there is always a nonzero probability for
discrete random variables that the maximum likelihood will not exist.  
For instance, if $a_i$ is an extreme ray of $\cone(A)$, then
the data vector $m a_i$ will always have a nonzero probability of occurring 
but the maximum likelihood estimate never exists for this vector of counts.
\end{Rem}

Since the weak maximum likelihood threshold has a nice interpretation in terms
of lattice points in the marginal cone, we can relate this concept to concepts in 
polyhedral geometry.  Let $\calp \subseteq \mathbb{R}^n$ be an integral polytope, that is, a polytope whose
vertices are all in $\mathbb{Z}^n$.  The $m$th dilate of $\calp$ is the set
$m \calp =  \{  mp  :  p \in \calp \}$.

\begin{Def}
 Let $\calp \subseteq \mathbb{R}^n$ be an integral polytope.  The \emph{codegree} of $\calp$,
 denoted $\mathrm{codeg}(\calp)$, is the smallest positive integer $m$ such that $m\calp$ has
 an integer point in its relative interior.
\end{Def}

The codegree is related to the weak maximum likelihood threshold because the 
$m$th dilates of $\mathrm{Marg}(A)$ contain all the sufficient statistics of
the data points $u$ of sample size $m$.

\begin{Th}\label{thm:codegwmlt}
 Let $A \in \mathbb{Z}^{k \times n}$ be an integer
matrix with  $\mathbf{1} \in \text{rowspan}(A)$.  
Then 
\[\mathrm{codeg}(\mathrm{Marg}(A)) \leq  \wmlt( \calm_A ).\]
\end{Th}

\begin{proof}
Note that  $\mathrm{Marg}(A)$ is the convex hull of all sufficient statistics vectors
$Au$ where $u$ is a data point with sample size one.  The $m$th dilate of $\mathrm{Marg}(A)$
is the convex hull of all $Au$ where $u$ is a data point of sample size $m$.  So, if there
is a data point $u$ of sample size $m$ such that $Au$ is in the relative interior of 
$\mathrm{cone}(A)$ then $Au$ is a lattice point in the relative interior of $m\mathrm{Marg}(A)$.
This shows that for there to exist a data point which is a witness to the weak maximum likelihood threshold being greater than or equal to $m$, there much be a lattice point in the relative interior
of $m \mathrm{Marg}(A)$ at the least, which implies this bound.
\end{proof}

The reason why Theorem \ref{thm:codegwmlt} is just a bound, and not sharp in general, is that 
there can be situations where $m \mathrm{Marg}(A)$ contains lattice points in its
relative interior but none of those points
are of the form $Au$ for any nonnegative integer vector $u$.
The situation where that does not happened is captured by the notion of normality.

\begin{Def}
 A matrix    $A \in \mathbb{Z}^{k \times n}$ is called \emph{normal} if 
 \[
{\rm cone}(A)  \cap \{  Au :  u \in \mathbb{Z}^n \}  =  \{  Au :  u \in \mathbb{N}^n \}. 
 \]
\end{Def}

A closely related notion is the notion of normality of a polytope.

\begin{Def}
A lattice polytope $\calp \subseteq \mathbb{R}^n$ is called \emph{normal} if every lattice point $v \in m \calp$ can be written as $v_1 + \dots + v_m$ for lattice points $v_1,\dots,v_m \in \calp$.
\end{Def}

Note that the two notions of normality coincide in the event that the lattice generated by the lattice points in $\calp = \conv(A)$ is equal to $\zz^n \cap \cone(A)$.  

\begin{Cor} \label{cor:normal}
     Let $A \in \mathbb{Z}^{k \times n}$ be an integer
matrix with  $\mathbf{1} \in \text{rowspan}(A)$ and suppose  that $A$ is normal.
Then 
\[\mathrm{codeg}(\mathrm{Marg}(A)) =  \wmlt( \calm_A ).\]
\end{Cor}

\begin{proof}
    If $A$ is normal and $m \Marg(A)$ contains an interior lattice point, then
    that lattice point is of the form $Au$ for some nonnegative integral vector $u$.
    Hence, if $m$ is the codegree of $\Marg(A)$ and $A$ is normal then $m \geq \wmlt(\calm_A)$.
\end{proof}

If $A$ is not normal, the inequality maybe be strict or tight, depending on the situation.
In the next section, the hierarchical model of the $K_4$ graph provides an
example of a matrix $A_{\Gamma, \mathbf{2}}$ where the codegree is strictly smaller
than the weak maximum likelihood threshold (Example \ref{ex:k4}).


\section{The Codegree of Hierarchical Models} \label{sec:codeghier}

In this section we introduce hierarchical models and their
marginal polytopes.  We prove some results about the codegree of hierarchical 
models including a general lower bound, that we show is tight in a number of 
cases.  We also discuss the known results on normality of hierarchical models
and use those results to get bounds on the weak maximum likelihood threshold
for hierarchical models.  

Hierarchical models are widely used statistical models that specify interactions
between collections of random variables.  The particular sets of random variables
that will have interaction terms in the model are determined by the faces
of an associated simplicial complex.  Data for these models is encoded in a 
multiway table, and the sufficient statistics of a hierarchical model are
marginals of the multiway table.  Since hierarchical models are
among the most important log-linear models, this is the reason  that the expressions
``marginal cone'' and ``marginal polytope'' are used in general for log-linear models.

\begin{Def}
Let $S$ be a finite set and let $2^S$ denote the set of all subsets of $S$. 
A \emph{simplicial complex} on $S$ is a set $\Gamma \subseteq 2^S$ such that for each $\sigma \in \Gamma$, if $\tau \subseteq \sigma$, then $\tau \in \Gamma$. 
\end{Def}

We typically denote the ground set by $|\Gamma|$, which is the set $S$.  This
notation will be useful when we have multiple simplicial complexes on overlapping
ground sets.
The elements of $\Gamma$ are called \emph{faces} of $\Gamma$ and faces that are maximal with respect to inclusion are called \emph{facets}. Let $\Gamma$ be a simplicial complex and let $d \in \mathbb{Z}_{\geq 2}^{|\Gamma|}$ be a vertex labeling. We associate to $\Gamma$ and $d$ a design matrix $A_{\Gamma,d}$ as follows.
\begin{itemize}
    \item Let $\mathcal{D} = \prod\limits_{t \in |\Gamma|} [d_t]$. The elements of $\mathcal{D}$ index the columns of $A_{\Gamma,d}$.
    \item For $\sigma \in \Gamma$, let $\mathcal{D}_\sigma = \prod\limits_{t \in \sigma} [d_t]$. The rows of $A_{\Gamma,d}$ are indexed by pairs $(\sigma, i_\sigma)$ where $\sigma \in \text{facets}(\Gamma)$ and $i_\sigma \in \mathcal{D}_\sigma$. 
    \item Given $j \in \mathcal{D}$ and $\sigma \in \text{facets}(\Gamma)$, let $j_\sigma$ be the coordinates of $j$ indexed by elements of $\sigma$. 
    \item  For each $j \in \mathcal{D}$ let $a^j$ be the column vector 
    whose coordinates $a^j_{(\sigma,i_\sigma)} $ satisfy 
    \[a^j_{(\sigma, i_\sigma)} = \begin{cases} 1 &~\mbox{ }~i_\sigma = j_\sigma \\
    0 &~\mbox{ }~\text{otherwise} \end{cases} \]
    \item  The design matrix $A_{\Gamma, d}$ is the matrix whose columns are
    the vectors $a^j$, for $j \in \mathcal{D}.$
\end{itemize}
The \emph{hierarchical model} $\mathcal{M}_{\Gamma,d}$ is the 
log-linear model with design matrix $A_{\Gamma,d}$. 
Let $\mathrm{Marg}(\Gamma,d)$ denote the marginal polytope of $\mathcal{M}_{\Gamma,d}$. 

\begin{Ex}
    Let $\Gamma = [12][23]$, the simplicial complex whose facets are $\{1,2\}$ and
    $\{2,3\}$ and let $d = (2,2,2)$.
    Then the design matrix of $\text{Marg}(\Gamma,d)$ is
\[
     \bordermatrix{ & 111 & 112 & 121 & 122 & 211 & 212 & 221 & 222   \cr
        ([12], 11) & 1 & 1 & 0 & 0 & 0 & 0 & 0 & 0  \cr
       ([12], 12)  & 0 & 0 & 1 & 1 & 0 & 0 & 0 & 0 \cr
       ([12], 21) & 0 & 0 & 0 & 0 & 1 & 1 & 0 & 0  \cr
       ([12], 22) & 0 & 0 & 0 & 0 & 0 & 0 & 1 & 1  \cr
       ([23], 11)  & 1 & 0 & 0 & 0 & 1 & 0 & 0 & 0 \cr
        ([23], 12) & 0 & 1 & 0 & 0 & 0 & 1 & 0 & 0 \cr
        ([23], 21)  & 0 & 0 & 1 & 0 & 0 & 0 & 1 & 0 \cr
        ([23], 22)  & 0 & 0 & 0 & 1 & 0 & 0 & 0 & 1 \cr}
    \]
\end{Ex}

The polyhedral geometry of the marginal polytope of hierarchical models
is a complicated subject, and not much is known about the facet defining inequalities 
of the marginal polytope in general.  
A complete description is probably impossible,
as it would include as a special case the cut polytopes, which are known
to be intractable \cite{Deza2010}.  
Some basic facts are known, and they are useful in
our study of the codegree.  For example:

\begin{Prop}
    Let $\Gamma$ be a simplicial complex and $d \in \mathbb{Z}_{\geq 2}^{|\Gamma|}$.
    Then for each $\sigma \in \mathrm{facet}(\Gamma)$ and each $i_\sigma \in \mathcal{D}_\sigma$,
    $x_{(\sigma,i_\sigma)} \geq 0$ is a facet defining inequality of the marginal cone
    $\mathrm{cone}(A_{\Gamma, d}).$
\end{Prop}

This result seems to be well-known, but we have not been able to find a specific reference
    for it in the literature.  Here is a sketch of an argument.

\begin{proof}
    The inequality $x_{(\sigma,i_\sigma)} \geq 0$ is clearly a valid inequality for the marginal cone.
    To see that it is facet defining, we must show that the columns of $A_{\Gamma,d}$ that
    that satisfy $x_{(\sigma,i_\sigma)} = 0$, span a space of dimension one less than the rank of 
    $A_{\Gamma,d}$.  
    
    We can assume without loss of generality that $i_\sigma = (d_j : j \in \sigma)$.  For a fixed simplicial complex $\Gamma$ and $d \in \mathbb{Z}_{\geq 2}^{|\Gamma|}$
    consider the set of columns of $A_{\Gamma,d}$:    
\[
B_{\Gamma, d} = \cup_{\sigma \in \Gamma}  \{ a^j :  j_k \in \{2, \ldots, d_k\} \mbox{ if } k \in \sigma 
\mbox{ and }  j_k = 1 \mbox{ if } k \not\in \sigma \}.
\]
Note that the union runs over all faces of $\Gamma$ (not just the facets). 
One verifies that
\begin{itemize}
    \item $B_{\Gamma, d}$ is linearly independent
    \item  The number of elements in $B_{\Gamma, d}$ equals the rank of $A_{\Gamma, d}$
    \item  All the elements of $B_{\Gamma,d}$ except $a^{j^*}$ satisfy $x_{(\sigma,i_\sigma)} \geq 0$
    (where $a^{j^*}$ is the unique vector in  $B_{\Gamma,d}$ with $j_\sigma = i_\sigma$).  
\end{itemize}
These facts imply that $x_{(\sigma,i_\sigma)} \geq 0$ is a facet defining inequality of the marginal cone
    $\mathrm{cone}(A_{\Gamma, d}).$
\end{proof}

\begin{Def}
Let $\Gamma$ be a simplicial complex on $[n]$ with  $d \in \mathbb{Z}^n_{\geq 2}$.
The \emph{weight} of a face $\sigma \in \Gamma$ is $\prod\limits_{k \in \sigma} d_k$. 
We denote the maximum weight among all facets of $\Gamma$ by $\omega(\Gamma,d)$.
\end{Def}

For the marginal polytope of a hierarchical model, the codegree is bounded by the weights of the faces of the associated simplicial complex.

\begin{Prop}\label{Prop:codegreelowerbound}
Let $\Gamma$ be a simplicial complex on $[n]$ and let $d \in \mathbb{Z}_{\geq 2}^{|\Gamma|}$ be a weight vector. Then $\mathrm{codeg}(\mathrm{Marg}(\Gamma,d)) \geq \omega(\Gamma,d)$.
\end{Prop}

\begin{proof}
Let $k = \mathrm{codeg}(\mathrm{Marg}(\Gamma,d))$ and let 
$y = \sum\limits_{j \in \mathcal{D}} \lambda_j ka^j$ be a lattice point in the 
relative interior of $k\mathrm{Marg}(\Gamma,d)$. 
Note that $y_{(\sigma, i_\sigma)} \geq 1$ since $y$ is an integer point and since $x_{(\sigma, i_\sigma)} \geq 0$ 
defines a face of $\mathrm{Marg}(\Gamma,d)$ for every $i_\sigma$. 

Fix a facet $\sigma$ of $\Gamma$. 
Note that each vector $a^j$ has exactly one $1$ among the entries $a^j_{\sigma, i_\sigma}$.
Thus, for any $x \in k \mathrm{Marg}(\Gamma,d)$, $\sum_{i_\sigma \in \mathcal{D}_\sigma} x_{(\sigma,i_\sigma)} = k$.
Now if $\sigma$ is a facet of maximum weight, then 
\[
\sum\limits_{i_\sigma \in \mathcal{D}_\sigma} 1  =  \omega(\Gamma, d)
\]
So for $x$ an integer point in the interior of $k \mathrm{Marg}(\Gamma,d))$
with $k  = \mathrm{codeg}(\mathrm{Marg}(\Gamma,d))$ we have:
\[
\mathrm{codeg}(\mathrm{Marg}(\Gamma,d)) = k =  \sum_{i_\sigma \in \mathcal{D}_\sigma} y_{(\sigma,i_\sigma)}
\geq \sum_{i_\sigma \in \mathcal{D}_\sigma} 1  =  \omega(\Gamma, d).
\]
Thus $\mathrm{codeg}(\mathrm{Marg}(\Gamma,d)) \geq \omega(\Gamma,d)$.
\end{proof}

\begin{Ex}
Let $\Gamma$ be the simplicial complex $[12][23]$ and let $d = (3,2,2)$. 
If the coordinates of $\Marg(\Gamma,d)$ are in the order
\[
(([12], 11), ([12], 12), ([12], 21),([12], 22),([12], 31),([12], 11),  \]
\[
([23], 11),([23], 12),([23], 21),([23], 22))
\]
then the following points lie in the relative interior of $6 \left(\text{Marg}(\Gamma,d)\right)$:
\[\begin{array}{ccc}
     (1,1,1,1,1,1,1,2,1,2), &(1,1,1,1,1,1,1,2,2,1), \\
     (1,1,1,1,1,1,2,1,1,2), &(1,1,1,1,1,1,2,1,2,1).
\end{array}\]
Each vertex of $\Marg(\Gamma,d)$ has one of the first six coordinates equal to $1$. 
As a consequence, the sum the first $6$ coordinates must equal the dilation factor. Each coordinate must be positive in a point that lies in the relative interior, and so this is the smallest dilate with a lattice point in the relative interior.

We note that because the weights of  the facets $[12]$ and $[23]$ are different,
there were multiple ways for the last 4 coordinates to sum to 6 
and still yield a lattice point in the relative interior of this polytope. 
When the weights of all facets of $\Gamma$ are equal, the lattice point is unique (which we show later in Lemma~\ref{lemma:codegreeMarg}).
\end{Ex}

In general, we conjecture that the bound from Proposition \ref{Prop:codegreelowerbound}
is tight.

\begin{Conj}\label{conj:codegreehier}
    Let $\Gamma$ be a simplicial complex on $[n]$ and let $d \in \mathbb{Z}_{\geq 2}^{|\Gamma|}$ be a weight vector. Then $\mathrm{codeg}(\mathrm{Marg}(\Gamma,d)) = \omega(\Gamma,d)$.
\end{Conj}

We prove Conjecture \ref{conj:codegreehier} in some special cases below.
The argument depends on 
producing explicit lattice points in the relative interior 
of $\omega(\Gamma, d) \Marg(\Gamma,d).$

\begin{Lemma}
\label{lemma:codegreeMarg}
Let $\Gamma$ be a simplicial complex on $[n]$ with weight vector $d \in \mathbb{Z}_{\geq 2}^n$.
If all facets have equal weight, then the codegree of $\emph{Marg}(\Gamma,d)$ is $\omega(\Gamma,d)$. 
Moreover, in this case $\mathbf{1}$ is the unique lattice point in the relative interior of $\omega(\Gamma, d) \Marg(\Gamma,d)$.
\end{Lemma}

\begin{proof}
Let $\mathcal{D} = \prod\limits_{i=1}^n [d_i]$. 
By Proposition \ref{Prop:codegreelowerbound} it suffices to show that there exists a unique lattice point in the relative interior of $\omega(\Gamma,d)\Marg(\Gamma,d)$. We take a convex combination of all the vertices of $\omega(\Gamma,d)\Marg(\Delta,d)$, with all vertices weighted equally with
weight $1/|\mathcal{D}|$.
Consider the $i_\sigma$ coordinate:
\[
\left(\sum\limits_{j \in \mathcal{D}} \frac{1}{|\mathcal{D}|} \omega(\Gamma,d) a^j\right)_{i_\sigma} 
= \frac{\omega(\Gamma,d)}{|\mathcal{D}|} \sum\limits_{\substack{j \in \mathcal{D} \\ i_\sigma = j_\sigma}} 1 
= \frac{\omega(\Gamma,d)}{|\mathcal{D}|} \times \frac{|\mathcal{D}|} {\omega(\Gamma,d)} = 1
\]
Since each vertex has nonzero weight in this convex combination, $\mathbf{1}$ lies in the relative interior of $\Marg(\Gamma,d)$. 
To see that it is unique, for each facet $\sigma$ consider the coordinate projection $\pi_\sigma:\omega(\Gamma,d)\Marg(\Gamma,d) \to \omega(\Gamma,d)\Marg(\sigma,d_\sigma)$. 
Given $v$ in the relative interior of $\Marg(\Gamma,d)$, 
$\pi_\sigma(v)$ lies in the relative interior of $\Marg(\sigma,d_\sigma)$. This marginal polytope is the probability simplex. Since all facets have equal weight $\omega(\Gamma,d) = \omega(\sigma,d_\sigma)$. A unique lattice point in the relative interior of $\omega(\Gamma,d)\Marg(\sigma,d_\sigma)$ must have positive integer coordinates summing to $\omega(\Gamma,d)$, and $\mathbf{1}$ is the only such lattice point. 
Consequently $\pi_\sigma(v) = \mathbf{1}$. But $\pi_\sigma$ is a coordinate projection, so $v = \mathbf{1}$. Hence $\omega(\Gamma,d) \Marg(\Gamma,d)$ has a unique lattice point in the relative interior.
\end{proof}

It is also possible to produce lattice points in the relative interior of 
$\omega(\Gamma, d) \Marg(\Gamma,d)$ in more general circumstances. For instance:

\begin{Lemma}
\label{lemma:codegdivide}
Let $\Gamma$ be a simplicial complex on $[n]$ with weight vector $d \in \mathbb{Z}_{\geq 2}^n$.
Suppose that for each facet $\sigma$ of $\Gamma$, the weight divides $\omega(\Gamma,d)$. Then the codegree of $\Marg(\Gamma,d)$ is $\omega(\Gamma,d)$.
\end{Lemma}

\begin{proof}
    The proof is the same as the first part of Lemma \ref{lemma:codegreeMarg}.
Let $\mathcal{D} = \prod\limits_{i=1}^n [d_i]$. 
By Proposition \ref{Prop:codegreelowerbound} it suffices to show that there exists a  lattice point in the relative interior of $\omega(\Gamma,d)\Marg(\Gamma,d)$. Let $v$ be the convex combination of all the vertices of $\omega(\Gamma,d)\Marg(\Delta,d)$, with all vertices weighted equally with
weight $1/|\mathcal{D}|$. 

Let $w_\sigma$ denote the weight of the face $F$.
Consider the $i_\sigma$ coordinate:
\[
\left(\sum\limits_{j \in \mathcal{D}} \frac{1}{|\mathcal{D}|} \omega(\Gamma,d) a^j\right)_{i_\sigma} 
= \frac{\omega(\Gamma,d)}{|\mathcal{D}|} \sum\limits_{\substack{j \in \mathcal{D} \\ i_\sigma = j_\sigma}} 1 
= \frac{\omega(\Gamma,d)}{|\mathcal{D}|}  \frac{|D|}{w_\sigma}  =  \frac{\omega(\Gamma,d)}{w_\sigma}.
\]
Since the weight $w_\sigma$ divides $\omega(\Gamma,d)$ this is always an integer, and hence 
$v$ is an integer point in $\omega(\Gamma,d)\Marg(\Gamma,d)$.  It is also in the relative
interior of $\omega(\Gamma,d) \Marg(\Gamma, d)$, so this shows that
$\mathrm{codeg}(\Marg(\Gamma, d)) = \omega(\Gamma,d)$.
\end{proof}

We can relate these results back to the weak maximum likelihood threshold of 
the hierarchical models.  So we deduce lower bounds for the amount of data that is
needed for there to be any chance of data point for which the maximum likelihood estimate
exists.  For example, if $\Gamma$ is a graph, and $d = {\bf 2}$ then a classification of
when $A_{\Gamma, {\bf 2}}$ is normal is known, and from this we get sharp results on the weak 
maximum likelihood threshold.

\begin{Cor}
    Let $\Gamma$ be a graph that is free of $K_4$ minors and contains an edge. 
    Then
    \[
\wmlt(A_{\Gamma, {\bf 2}}) = 4.
    \]
\end{Cor}

\begin{proof}
    Lemma \ref{lemma:codegdivide} shows that the codegree of $A_{\Gamma, {\bf 2}} = 4$ ,
    for any graph $\Gamma$ that contains an edge.  
    The main result of \cite{Sullivant2010} is that $A_{\Gamma, {\bf 2}}$
    is normal if and only if $\Gamma$ does not have a $K_4$ minor.
    Hence by Corollary \ref{cor:normal}, $\wmlt(A_{\Gamma, {\bf 2}}) = 4.$
\end{proof}

\begin{Ex}  \label{ex:k4}
Note that the example of $\Gamma = K_4$ is an interesting one from the standpoint
of the weak maximum likelihood threshold.  In that case, 
$4 \Marg(K_4, {\bf 2})$ contains exactly one interior lattice point, but it is
not a sum of $4$ lattice points in $\Marg(K_4, {\bf 2})$.  This lattice point is the
unique hole in the semigroup generated by $A_{K_4, {\bf 2}}$  (\cite{Hemmecke2009}), so we have that   $\wmlt(A_{K_4, {\bf 2}}) = 5$.
\end{Ex}

\begin{Ques}
What is the weak maximum likelihood threshold of the graphs $\Gamma$ that do
contain a $K_4$ minor?  Can the gap between the degree and the weak maximum likelihood
threshold be arbitrarily large in this case?
\end{Ques}


\section{Fiber Products and Gorenstein Decomposable Models} \label{sec:Gorenstein}

In this section we study Gorenstein hierarchical models,
and in particular classify the decomposable models that are Gorenstein.
A key tool is a result of Dinu and Vodi\v{c}ka \cite{Dinu2021} which
give a description of when the fiber product
of polytopes preserves the Gorenstein property.  
The fiber product of polytopes is closely related to
the toric fiber product construction \cite{Sullivant2007}.

\begin{Def}
    An integer polytope $\calp$ is called \emph{Gorenstein} of index $k$ if $\calp$ is normal,
    $kP$ has a unique interior lattice point $v$, and, for each facet of $k\calp$, $v$ is  lattice
    distance $1$ away from the facet.  
\end{Def}

The condition that $v$ is lattice distance $1$ from the facet can be expressed as follows.
If $ax \leq b$ is a facet defining inequality of $\calp$ with 
$a \in \mathbb{Z}^n$ and $\gcd(a_1, \ldots, a_n) = 1$,
then $av  =  b-1$.

Gorenstein cut polytopes were classified in \cite{Ohsugi2014}.  
As a consequence of this classification,
we also get a classification of Gorenstein hierarchical models when the underlying
simplicial complex $\Gamma$ is a graph, and $d = {\bf 2}$.  Recall that
a graph is called \emph{chordal} if it does not have any induced cycles of
length $\geq 4$.

\begin{Th}\label{thm:ohsugicor}
   Let $\Gamma$ be a graph that contains an edge.  
   Then $\Marg(\Gamma, {\bf 2} )$ is Gorenstein if and only if
   $\Gamma$ is chordal, free of $K_4$ minors, and has no isolated vertices.
\end{Th}

\begin{proof}
    Here we use the result of \cite{Ohsugi2014} which classified Gorenstein cut polytopes.
    Cut polytopes $\mathrm{Cut}^\square(G)$ are associated to graphs $G$.  
    If $\Gamma$ is a graph then $\Marg(\Gamma, {\bf 2}) \cong \mathrm{Cut}^\square(\hat\Gamma)$
    where $\hat\Gamma$ is the suspension of $\Gamma$, obtained by taking $\Gamma$ and
    adding a new vertex which $v_0$ and edges from $v_0$ to every vertex of $\Gamma$.
    
    In Ohsugi's classification of Gorenstein cut polytopes, there are two types of
    graphs that  yield Gorenstein cut polytope.  These are
    \begin{enumerate}
        \item Bipartite graphs, free of $C_6$ minors and, $K_5$ minors
        \item Bridgeless chordal graphs, free of $K_5$ minors.
     \end{enumerate}
    The first type of graph cannot appear as the suspension of a graph with an edge (since
    such a suspension must have triangles and hence cannot be bipartite).  
    A suspension of a graph
    yields a graph of the second type if and only if it is a chordal graph with no $K_4$ minors
    and no isolated vertices.
\end{proof}    

Our goal in this section is to develop further cases where we can prove
that the Gorenstein property is satisfied for hierarchical models.  
A useful observation is related to the weight of facets.

\begin{Prop}
\label{prop:sameWeight}
    Suppose that $\Marg(\Gamma, {\bf d} )$ is Gorenstein.  Then all facets of $\Gamma$ have the same weight.      
\end{Prop}

\begin{proof}
    Since the inequalities $x_{(\sigma, i_\sigma)} \geq 0$ are facet defining, the only point that is possible
    to be lattice distance one from each facet in some dilate is the point ${\bf 1}$.  However, ${\bf 1}$ is
    in the marginal cone if and only if all facets of $\Gamma$ have the same weight.
\end{proof}

\begin{Cor}
   If $\Marg(\Gamma, {\bf 2} )$ is Gorenstein then $\Gamma$ is pure.
\end{Cor}

\begin{proof}
    If ${\bf d} = {\bf 2}$ then the weight of a facet $\sigma$  is $2^{\# \sigma} $. For all these weights to be 
    equal, all facets must have the same dimension.
\end{proof}

One simple situation where we know that the Gorenstein property is preserved is when
a simplicial complex is a cone over another simplicial complex.  

\begin{Def}
    Let $\Gamma$ be a simplicial complex on $[n]$.  Let $C(\Gamma)$ denote the new simplicial complex with one
    more vertex $n+1$ with faces $C(\Gamma) = \Gamma \cup \{  \sigma \cup \{n+1 \} :  \sigma \in \Gamma \}.$
    This new simplicial complex is the \emph{cone over $\Gamma$}.
    The complex $C^r(\Gamma)$ is obtained by iterating the cone construction $r$ times.
\end{Def}

\begin{Prop}
    Let $\Gamma$ be a simplicial complex.   Let
    $d'$ be the vector $d$ with $d_{n+1}$  appended.
    Then $\Marg(\Gamma, d )$ is Gorenstein of index $k$ if and only if
     $\Marg(C(\Gamma), d' )$ is Gorenstein of index $k d_{n+1}$.
\end{Prop}

\begin{proof}
    After rearranging rows and columns of $A_{C(\Gamma), d'}$, we see that this matrix has a block diagonal form isomorphic
    to
    \[
    A_{C(\Gamma), d'}  =  \begin{pmatrix}
    A_{\Gamma, d } &  0  & \cdots & 0 \\
    0  &  A_{\Gamma, d }  & \cdots  & 0 \\
    \vdots  & \vdots  & \ddots & \vdots  \\
    0 & 0 &   \cdots  &  A_{\Gamma, d } 
    \end{pmatrix}
    \]
    The facet defining inequalities for this are just repeated copies of the facet defining inequalities for $\Marg(\Gamma, d)$.
    If $v$  is the unique lattice point in the interior of $k \Marg(\Gamma, d )$  then 
    $(v, \ldots, v)$ is the unique lattice point in the interior  of $k d_{n+1} \Marg(C(\Gamma), d' )$ and it is
    distance one from all facets since $v$ is.
\end{proof}

A key situation
where we will get new results concerns reducible and decomposable hierarchical models.

\begin{Def}
    A simplicial complex $\Gamma$ is \emph{reducible} with reducible decomposition
    $(\Gamma_1, S, \Gamma_2)$ and separator $S \subseteq |\Gamma|$ if 
    $\Gamma = \Gamma_1 \cup \Gamma_2$, $\Gamma_1 \cap \Gamma_2 = 2^S$, and
    $\Gamma_1$ and $\Gamma_2$ are proper subsets of $\Gamma$.
\end{Def}

\begin{Def}
    A simplicial complex $\Gamma$ is \emph{decomposable} if it is either a simplex $2^S$,
    or it has a reducible decomposition $(\Gamma_1, S, \Gamma_2)$ where both $\Gamma_1$
    and $\Gamma_2$ are decomposable.
\end{Def}

Our main theorem in this section classifies the Gorenstein decomposable
hierarchical models.

\begin{Th}\label{thm:decomposable}
    Let $\Gamma$ be a decomposable simplicial complex and $d \in \mathbb{Z}_{\geq 2}^n$.
    Then $\Marg(\Gamma, d)$ is Gorenstein if and only if every facet of $\Gamma$ has the
    same weight $\omega(\Gamma, d)$.
\end{Th}

A special case of Theorem \ref{thm:decomposable}, is when $d = {\bf 2}$, in which case
we see that a decomposable binary hierarchical model is Gorenstein if and only if $\Gamma$ is
pure (that is, all facets have the same dimension). 

One of the main tools we use to prove this is the fiber product of polytopes. 
Let $\calp_1 \subseteq \mathbb{R}^{d_1}$, $\calp_2 \subseteq \mathbb{R}^{d_2}$, 
and $\calq \subseteq \mathbb{R}^e$ be polytopes. 
The \emph{product polytope} of $\calp_1$ and $\calp_2$ is 
$\calp_1 \times \calp_2 = \{(x,y): x \in \calp_1 \text{ and } y \in \calp_2\}$.

\begin{Def} 
Let $\pi_1:\calp_1 \to \calq$ and $\pi_2:\calp_2 \to \calq$ be affine maps. 
The \emph{fiber product} of $\calp_1$ and $\calp_2$ with respect to the maps $\pi_1$ and $\pi_2$ is
\[
\calp_1 \times_\calq \calp_2 = \{(x,y) \in \calp_1 \times \calp_2: \pi_1(x) = \pi_2(y)\}.
\]
\end{Def}
We suppress the maps $\pi_1$ and $\pi_2$ from the notation. If $\pi_1(\calp_1) = \pi_2(\calp_2) = \{c\}$ for some $c \in \calq$, then $\calp_1 \times_\calq \calp_2 = \calp_1 \times \calp_2$. 

The toric fiber product \cite{Sullivant2007} is an important tool for studying 
ideals that arise in algebraic statistics.  The fiber product of polytopes is the resulting
combinatorial construction that arises from looking at the associated moment polytope
associated to the grading groups that arise.  Some relevant papers include 
\cite{Dinu2021, Engstrom2014, Rauh2016}.

When polytopes and their vertices are indexed by combinatorial structures, the fiber product of polytopes often corresponds to ``gluing rules'' for these structures. We see this in the following example.

\begin{Ex}
Let $\Gamma_1 = [12]$, $d^1 = (3,2)$, $\Gamma_2 = [23]$, and $d^2 = (2,2)$. 
The simplicial complexes $\Gamma_1$ and $\Gamma_2$ 
intersect in the simplex $S = [2]$ and have common vertex labeling $d^S = [2]$. 
We define marginalization maps $\pi_1:\Marg(\Gamma_1,d^1) \to \Marg(S,d^S)$ and 
$\pi_2: \Marg(\Gamma_2,d^2) \to \Marg(S,d^S)$ defined by
\[
\pi_1(x)_{(i_2,2)} = \sum\limits_{j \in [3] \times \{i_2\}} x_j \quad \text{and} \quad \pi_2(x) = \sum\limits_{j \in \{i_2\} \times [2]} x_j.
\]
Let $\Gamma = [12][23]$ and let $d = (3,2,2)$. 
For $i \in [d_1] \times [d_2] \times [d_3]$ we have
\[
\pi_1(a^{i_1i_2}) = u^{i_2} = \pi_2(a^{i_2i_3}).
\]
The fiber product is $\Marg(\Gamma_1,d^1) \times \Marg(\Gamma_2,d^2)$ 
intersected with the affine subspace $\pi_1(x) = \pi_2(y)$ for $(x,y) \in 
\Marg(\Gamma_1,d^1) \times \Marg(\Gamma_2,d^2)$. 
Computing the vertices of this new polytope reveals that this fiber product is $\Marg(\Gamma,d)$.

Note that $\Gamma = \Gamma_1 \cup \Gamma_2$ and that $d$ is the vertex labeling resulting from this gluing, which is well-defined since $d^1$ and $d^2$ agree on the label of $2$. This gluing is also evident in the state vectors indexing the vertices. For example, consider
\[
a^{32} = (0,0,0,0,0,1) \quad \text{and} \quad a^{21} = (0,0,1,0)
\]
where the coordinates of $a^{32}$ and $a^{21}$ 
are ordered $(11,12,21,22,31,32)$ and $(11,12,21,22)$ 
respectively. Both vectors map to $u^2$ under their marginalization maps 
so their fiber product yields a point in the fiber product: $(0,0,0,0,0,1,0,0,1,0)$.
\end{Ex}

More generally, we have the following which is our motivation for 
developing the fiber product in more detail.

\begin{Prop}
    Let $\Gamma$ be a reducible simplicial complex with decomposition
    $(\Gamma_1, S, \Gamma_2)$, and $d \in \mathbb{Z}^n_{\geq 2}$  
    Let $d^1 = (d_i)_{i \in |\Gamma_1|}$ and $d^2 = (d_i)_{i \in |\Gamma_2|}$.  
    Let $\pi_i : \Marg(\Gamma_i, d^i) \rightarrow \Marg(S, d_S) $  be the marginalization
    map onto the face $S$.
    Then, $\Marg(\Gamma, d)$ is the fiber product of $\Marg(\Gamma_1, d^1)$ and
$\Marg(\Gamma_2, d^2)$ with respect to $\pi_1$ and $\pi_2$.
\end{Prop}

See Theorem 22 in \cite{Sullivant2007} for a proof.

For the remainder of this subsection
let $\calp_1 \subseteq \mathbb{R}^{d_1}$, $\calp_2 \subseteq \mathbb{R}^{d_2}$, 
$\calq \subseteq \mathbb{R}^e$ be polytopes with affine maps 
$\pi_1:\mathbb{R}^{d_1} \to \mathbb{R}^e$ and $\pi_2:\mathbb{R}^{d_2} \to \mathbb{R}^e$. 

A key tool  for proving the Gorenstein property of hierarchical models under fiber product 
structures is the following result of Dinu and Vodi\v{c}ka \cite{Dinu2021}.

\begin{Th}\label{thm:dinu}
Let  $ \calp_1  $  and $ \calp_2  $ be two Gorenstein polytopes with the same index $k$. 
Let $\pi_i:  \calp_i \rightarrow \rr^e$ be projections such that $\pi_1(\calp_1) = \pi_2(\calp_2) = \Delta_e$,
the standard simplex in $\rr^e$.
Let $p_1$ and $p_2$ be the unique interior lattice points of $k\calp_1$ and $k\calp_2$ respectively.
Suppose that $\pi_1(p_1) = \pi_2(p_2)$.  Then the fiber product $\calp_1 \times_{\Delta_e} \calp_2$ is also 
Gorenstein of index $k$.
\end{Th}

We can now consider the application of Theorem \ref{thm:dinu}
for hierarchical models.   Note that, if $\Marg(\Gamma,  d)$ is Gorenstein, then,
as discussed previously, the special interior lattice point in the interior of the appropriate dilate
is the vector ${\bf 1}$.  Suppose that $\sigma$ is a face of $\Gamma$ and let $\tau$ be any facet of $\Gamma$ that contains
$\sigma$.
Let $\pi:  \Marg(\Gamma, {\bf d}) \rightarrow \rr^\sigma$ be the marginal projection map.
Then $\pi({\bf 1})  =  k {\bf 1}$ where $k =  \prod_{i \in \tau \setminus \sigma} d_i$ (note that $k$ is independent of chosen $\tau$ by Proposition~\ref{prop:sameWeight}).  
This observation immediately implies the following about the preservation of the Gorenstein
property for reducible hierarchical models.

\begin{Th}\label{thm:reducible}
    Suppose that $\Gamma$ is a reducible simplicial complex with decomposition $(\Gamma_1, S, \Gamma_2)$ 
    and $d \in \zz_{\geq 2}^n$ such that all
    facets have the same weight.
    Then  $\Marg(\Gamma_1, d^1)$ and $\Marg(\Gamma_2, d^2)$ are Gorenstein if and only if
     $\Marg(\Gamma, d)$ is Gorenstein.
\end{Th}

\begin{proof}
   The forward direction immediately follows from Theorem \ref{thm:dinu}.  
   Note that generally, it is not true that if $\calp_1 \times_{\Delta_e} \calp_2$ then
   $\calp_1$ and $\calp_2$ are Gorenstein.   Furthermore, if they are Gorenstein, they do not all need to
   have the same index.  However, in the case of hierarchical models, since the special interior point
   is always ${\bf 1}$, and the index always needs to be the weight of the complex, this implies the converse.
\end{proof}

An immediate Corollary of Theorem \ref{thm:reducible} is the classification of decomposable
Gorenstein hierarchical models.

\begin{Cor}
    Let $\Gamma$ be a decomposable simplicial complex and $d \in \zz_{\geq 2}^n$ .
    Then $\Marg(\Gamma, d)$ is Gorenstein if and only if all facets have the same weight.
\end{Cor}

\begin{proof}
    We already know that all facets having the same weight is a necessary condition for being Gorenstein.
    We can use induction on the number of facets and Theorem \ref{thm:reducible} to deduce the
    result.  Note in particular, that if $\Gamma$ just has a single facet $[n]$, then $\Marg(\Gamma, d)$
    is a standard simplex and is Gorenstein of index the weight of $\Gamma$.  Then, if $\Gamma$ is decomposable
    and all facets have the same weight, $\Gamma$ has a reducible decomposition where the two constituent complexes
    $\Gamma_1$ and $\Gamma_2$ are also decomposable, have the same weight, and hence are Gorenstein by induction.  Applying Theorem
    \ref{thm:reducible} shows that $\Marg(\Gamma, d)$ is also Gorenstein.
\end{proof}

Using the results we have so far about properties of the simplicial complexes that preserve the Gorenstein
property, we can make the following conjecture for binary hierarchical models, that
generalizes Theorem \ref{thm:ohsugicor}  for graphs.

Let $\partial \Delta_n$ denote the boundary of an $n-1$ simplex, the simplicial complex on $[n]$ which has all faces
but $[n]$.

\begin{Conj}\label{conj:genohsugi}
    Let $\Gamma$ be a pure simplicial complex, whose facets all have $n$ elements.  Suppose that $\Marg(\Gamma, {\bf 2})$ is Gorenstein.  
    Then either $\Gamma$ is reducible or $\Gamma$ is one of the following complexes:  $\Delta_n$ or $C^i (\partial \Delta_{n+1 -i})$ for $i = 0, 1, \ldots, n-2$.
\end{Conj}

Another way to say this classes of complexes in Conjecture \ref{conj:genohsugi} is that they are the ones
that are built from reducible decompositions starting with the base complexes on the list 
$\Delta_n$ or $C^i (\partial \Delta_{n+1 -i})$ for $i = 0, 1, \ldots, n-2$.
Note that Conjecture \ref{conj:genohsugi}, includes the case of graphs as a special case.  In the case of graphs, it says
that every Gorenstein binary marginal polytope where $\Gamma$ is a graph can be built from gluing edges and triangles in 
reducible decompositions.  Conjecture \ref{conj:genohsugi}, generalizes this to arbitrarily dimensions, with a
longer list of base complexes obtained by the cone operations.

Note that all of the complexes $\Gamma$ that are constructed via Conjecture \ref{conj:genohsugi}  are Gorenstein, provided
that we have the following result.

\begin{Prop}
    The marginal polytope $\Marg(\partial\Delta_n, {\bf 2})$ is Gorenstein. 
\end{Prop}

\begin{proof}
    Note that $\Marg(\partial\Delta_n, {\bf 2})$ is unimodular \cite{Bernstein2017}, hence normal.  But $\Marg(\partial\Delta_n, {\bf 2})$
    has codimension $1$, so it must be Gorenstein.
\end{proof}

Thus we know that all the complexes from Conjecture \ref{conj:genohsugi} are Gorenstein, so proving the conjecture amounts to showing that
these are the only such complexes.  What is likely needed to make progress on that conjecture 
are new families of facet defining inequalities for hierarchical models that can be used to verify that
a marginal polytope is not Gorenstein.

As evidence towards Conjecture \ref{conj:genohsugi}, we can verify the Gorenstein property computationally in complexes with
small numbers of vertices. 

\begin{Prop}
    Conjecture \ref{conj:genohsugi} is true for all complexes on $6$ or fewer vertices.
\end{Prop}

\begin{proof}
We can verify the conjecture by checking the Gorenstein condition for all complexes on $6$ or fewer vertices.
In fact, given all that we know about normal complexes and graphs we can greatly reduce the number of cases.
    
For pure complexes on $4$ or fewer vertices, they are either graphs, in which case the Gorenstein property
is verified by Theorem \ref{thm:ohsugicor}, or 2 or 3 dimensional complexes, in which they must be one of the complexes
 $\Delta_4$ or $C^i (\partial \Delta_{4 -i})$ for $i = 0, 1, 2$, which are all Gorenstein.

For complexes on $5$ vertices, we refer to the calculations and classifications of normal complexes that were
performed in \cite{Bernstein2017}.  Note that on $5$ vertices, every $3$ dimensional complex that is pure has the form 
$C^i (\partial \Delta_{5 -i})$ for $i = 0, 1, 2, 3, 4$, which are all Gorenstein.
The $0$ and $1$ dimensional Gorenstein complexes are classified by Theorem \ref{thm:ohsugicor}.  So, it remains to classify the
$2$ dimensional Gorenstein complexes.  According to Section 7.2 of \cite{Bernstein2017},
the only pure $2$ dimensional complexes that are not reducible, or a cone over a normal $1$ dimensional complexes, and the
complexes
\begin{enumerate}
    \item $[123][124][135][245]$
    \item  $[123][124][134][235][245]$
    \item  $[123][124][134][235][245][345]$.
\end{enumerate}
These three complexes were verified to be not Gorenstein using \cite{sagemath}.

A similar calculation works for the pure simplicial complexes on $6$ vertices.  There are $24$ normal pure simplicial complexes
that are not reducible, and that are not of the form $C^i (\partial \Delta_{6 -i})$, for $i = 0, \ldots, 5$.  
They are the following complexes in the Table in Section 7.3 of \cite{Bernstein2017}:
6, 10,17,20,28,35,39,40,41,43,44,45,53,54,58,59,60,61,68,69,70,73,76,78.   All were verified 
to be not Gorenstein using \cite{polymake, sagemath}.
\end{proof}

\bibliography{gorenstein.bib}{}

\begin{thebibliography}{10}

\bibitem{Bernstein2021}
Daniel~Irving Bernstein, Sean Dewar, Steven~J. Gortler, Anthony Nixon, Meera
  Sitharam, and Louis Theran.
\newblock Maximum likelihood thresholds via graph rigidity, 2021.
\newblock https://arxiv.org/abs/2108.02185.

\bibitem{Bernstein2017}
Daniel~Irving Bernstein and Seth Sullivant.
\newblock Unimodular binary hierarchical models.
\newblock {\em J. Combin. Theory Ser. B}, 123:97--125, 2017.

\bibitem{Deza2010}
Michel~Marie Deza and Monique Laurent.
\newblock {\em Geometry of cuts and metrics}, volume~15 of {\em Algorithms and
  Combinatorics}.
\newblock Springer, Heidelberg, 2010.
\newblock First softcover printing of the 1997 original [MR1460488].

\bibitem{Dinu2021}
Rodica Dinu and Martin Vodi\v{c}ka.
\newblock Gorenstein property for phylogenetic trivalent trees.
\newblock {\em J. Algebra}, 575:233--255, 2021.

\bibitem{Engstrom2014}
Alexander Engstr\"{o}m, Thomas Kahle, and Seth Sullivant.
\newblock Multigraded commutative algebra of graph decompositions.
\newblock {\em J. Algebraic Combin.}, 39(2):335--372, 2014.

\bibitem{polymake}
Ewgenij Gawrilow and Michael Joswig.
\newblock polymake: a framework for analyzing convex polytopes.
\newblock In {\em Polytopes---combinatorics and computation ({O}berwolfach,
  1997)}, volume~29 of {\em DMV Sem.}, pages 43--73. Birkh\"{a}user, Basel,
  2000.

\bibitem{GrossSullivant2018}
Elizabeth Gross and Seth Sullivant.
\newblock The maximum likelihood threshold of a graph.
\newblock {\em Bernoulli}, 24(1):386--407, 2018.

\bibitem{Hemmecke2009}
Raymond Hemmecke, Akimichi Takemura, and Ruriko Yoshida.
\newblock Computing holes in semi-groups and its applications to transportation
  problems.
\newblock {\em Contrib. Discrete Math.}, 4(1):81--91, 2009.

\bibitem{Ohsugi2014}
Hidefumi Ohsugi.
\newblock Gorenstein cut polytopes.
\newblock {\em European J. Combin.}, 38:122--129, 2014.

\bibitem{Rauh2016}
Johannes Rauh and Seth Sullivant.
\newblock Lifting {M}arkov bases and higher codimension toric fiber products.
\newblock {\em J. Symbolic Comput.}, 74:276--307, 2016.

\bibitem{Sullivant2007}
Seth Sullivant.
\newblock Toric fiber products.
\newblock {\em J. Algebra}, 316(2):560--577, 2007.

\bibitem{Sullivant2010}
Seth Sullivant.
\newblock Normal binary graph models.
\newblock {\em Ann. Inst. Statist. Math.}, 62(4):717--726, 2010.

\bibitem{Sullivant2018}
Seth Sullivant.
\newblock {\em Algebraic statistics}, volume 194 of {\em Graduate Studies in
  Mathematics}.
\newblock American Mathematical Society, Providence, RI, 2018.

\bibitem{sagemath}
{The Sage Developers}.
\newblock {\em {S}ageMath, the {S}age {M}athematics {S}oftware {S}ystem
  ({V}ersion x.y.z)}, YYYY.
\newblock {\tt https://www.sagemath.org}.

\end{thebibliography}
\bibliographystyle{plain}


\end{document}